\documentclass[12pt]{article}
\usepackage{a4,tikz,amsthm,amsmath,enumerate,pgfmath}

\title{\textbf{A note on edge-colourings avoiding rainbow $K_4$ and
    monochromatic $K_m$}}
\author{Veselin Jungi\'{c}$^1$\\
  Tom\'{a}\v{s} Kaiser$^2$\footnote{The Institute for Theoretical
    Computer Science (ITI) is supported by project 1M0545 of the Czech
    Ministry of Education.}
  \\
  Daniel Kr\'{a}l'$^{3*}$} \date{}

\newtheorem{theorem}{Theorem}
\newtheorem{lemma}[theorem]{Lemma}
\newtheorem{proposition}[theorem]{Proposition}

\newtheorem{problem}{Problem}

\newcommand\Setx[1] {\left\{{#1}\right\}}

\newcommand\size[1] {\left|{#1}\right|}

\newcommand\maxr[3] {{maxR}({#1},{#2},{#3})}
\newcommand\minr[3] {{minR}({#1},{#2},{#3})}

\newcommand{\graph}{
  \tikzstyle{every node}=[circle,fill,minimum size=4pt,inner sep=0pt];
  \tikzstyle{every edge}=[draw];
  \tikzstyle{highlight}=[thick,draw];
}

\begin{document}

\footnotetext[1]{Department of Mathematics, Simon Fraser University,
  Burnaby, B.C., V5A 2R6, Canada. E-mail: \texttt{vjungic@sfu.ca}.}
\footnotetext[2]{Department of Mathematics and Institute for
  Theoretical Computer Science, University of West Bohemia,
  Univerzitn\'\i~8, 306~14~Plze\v n, Czech Republic. E-mail:
  \texttt{kaisert@kma.zcu.cz}. Supported by Research Plan MSM
  4977751301 of the Czech Ministry of Education.}
\footnotetext[3]{Institute for Theoretical Computer Science, Faculty
  of Mathematics and Physics, Charles University, Malostransk\'{e}
  n\'{a}m\v{e}st\'{\i} 25, 118~00~Prague, Czech Republic. E-mail:
  \texttt{kral@kam.mff.cuni.cz}.}%

\maketitle

\begin{abstract}
  We study the mixed Ramsey number $\maxr n {K_m}{K_r}$, defined as
  the maximum number of colours in an edge-colouring of the complete
  graph $K_n$, such that $K_n$ has no monochromatic complete subgraph
  on $m$ vertices and no rainbow complete subgraph on $r$
  vertices. Improving an upper bound of Axenovich and Iverson, we show
  that $\maxr n {K_m}{K_4} \leq n^{3/2}\sqrt{2m}$ for all $m\geq
  3$. Further, we discuss a possible way to improve their lower bound
  on $\maxr n {K_4}{K_4}$ based on incidence graphs of finite
  projective planes.
\end{abstract}


\section{Introduction}
\label{sec:intro}

A subgraph of an edge-coloured graph is \emph{monochromatic} if all of
its edges receive the same colour, and it is \emph{rainbow} if all the
edge colours are distinct. Ramsey theory was born with the observation
that every sufficiently large complete graph whose edges are coloured
by $k$ colours, where $k$ is a fixed integer, contains a large
monochromatic complete subgraph~\cite{bib:Ram-problem,
  bib:ES-combinatorial}. In the following decades, it evolved into a
rich part of graph theory with strong links to combinatorial number
theory~\cite{bib:Nes-ramsey} and combinatorial geometry (see,
e.g.,~\cite{bib:PA-combinatorial}). There are many Ramsey-type
problems that involve monochromatic substructures of various
combinatorial structures, but some are of a different type.

Among them is the question asked by Erd\H{o}s, Simonovits and
S\'{o}s~\cite{bib:ESS-anti}: Given a graph $H$ and an integer $n$,
what is the maximum number of colours in an edge-colouring of a
complete graph $K_n$ such that no copy of $H$ in $K_n$ is rainbow?
This number is the \emph{anti-Ramsey number} (for $H$ and $n$) (see
also \cite{bib:MBNL-anti-ramsey,bib:MBNL-anti-ramsey-2}).

As a combination of the two problems, Axenovich and
Iverson~\cite{bib:AI-edge-colorings} defined the \emph{mixed Ramsey
  numbers} $\maxr n G H$ and $\minr n G H$ as the maximum
(respectively, minimum) number of colours in an edge-colouring of
$K_n$ such that no monochromatic subgraph of $K_n$ is isomorphic to
$G$ and no rainbow subgraph is isomorphic to $H$. They noted that the
numbers are well-defined whenever the edges of $G$ do not induce a
star and $H$ is not a forest (see also \cite{bib:JW-pattern}). Their
results asymptotically determine the behaviour of $\maxr n G H$ in
most cases and exhibit a close relation between this number and the
\emph{vertex arboricity} $a(H)$ of $H$, defined as the least number of
parts in a decomposition of $V(H)$ into sets inducing acyclic
subgraphs of $H$.

In the present paper, we will be concerned with bounds on $\maxr n
{K_m}{K_4}$ for $m\geq 3$. Let us briefly recall some of the results
of~\cite{bib:AI-edge-colorings} related to $\maxr n G H$. Assume that
the edges of $G$ do not induce a star. If $a(H) \geq 3$, then $\maxr n
G H$ is quadratic in $n$, namely
\begin{equation*}
  \maxr n G H = \frac{n^2}2 \Bigl( 1-\frac1{a(H) - 1} \Bigr)(1+o(1)).
\end{equation*}
On the other hand, if $a(H) = 2$, then $\maxr n G H$ is subquadratic:
\begin{equation*}
  \maxr n G H = O(n^{2-\frac1\epsilon}),
\end{equation*}
for some $\epsilon$ which depends on $G$ and $H$. There exists a more
explicit upper bound if $V(H)$ can be decomposed into two sets
inducing forests, one of which is of order at most 2. In this case,
\begin{equation}
  \maxr n G H \leq n^{5/3} (1+o(1)). \label{eq:upper}
\end{equation}
In the special case that $H$ is a cycle, $\maxr n G H$ can be
determined for non-bipartite graphs $G$:
\begin{equation*}
  \maxr n G {C_k} = n\Bigl( \frac{k-2}2 + \frac1{k-1} \Bigr) + O(1).
\end{equation*}
As for the lower bounds, Axenovich and
Iverson~\cite{bib:AI-edge-colorings} prove that if $G$ is
non-bipartite and the minimum degree of $H$ is at least $3$, then
\begin{equation}
  \maxr n G H \geq n \log n. \label{eq:lower}
\end{equation}

If we restrict to the case where $G$ and $H$ are complete graphs, the
above results asymptotically determine $\maxr n G H$ in all cases
except $H = K_4$, where we only have the bounds (\ref{eq:upper}) and
(\ref{eq:lower}). In particular, the problem of determining $\maxr n
{K_4} {K_4}$ is referred to in~\cite{bib:AI-edge-colorings} as `one of
the most intriguing' in this area.

The purpose of this note is to improve the upper bound on $\maxr n
{K_m} {K_4}$ for all $m\geq 3$:
\begin{theorem}\label{t:new-upper}
  \begin{equation*}
    \maxr n {K_m} {K_4} \leq n^{3/2} \sqrt{2m}.
  \end{equation*}
\end{theorem}

We prove this bound in Section~\ref{sec:upper}. In
Section~\ref{sec:lower}, we discuss a possible way to improve the
lower bound (for $m=4$) based on incidence graphs of finite projective
planes, and present some open problems.


\section{The upper bound}
\label{sec:upper}

Let $m\geq 3$ be a fixed integer throughout this section. Let us call
an edge-colouring of a complete graph $K_n$ \emph{admissible} if $K_n$
has no monochromatic complete subgraph on $m$ vertices and no rainbow
complete subgraph on $4$ vertices. For an admissible edge-colouring
$c$ of $K_n$ and disjoint sets $A,B\subset V(K_n)$, we define
$S_c(A,B)$ as the set of colours that are used by $c$, but only on
edges joining $A$ to $B$. Furthermore, we set $\sigma_c(A,B) =
\size{S_c(A,B)}$.

To prove the following lemma, one could use a suitable version of the
Zaran\-kie\-wicz theorem (e.g., that in \cite[Exercise
2.6]{bib:Juk-extremal}). For the reader's convenience, we give a
self-contained proof.

\begin{lemma}\label{l:across}
  Let $c$ be an admissible edge-colouring of $K_n$ and $A,B$ disjoint
  subsets of $V(K_n)$ each of size at most $k$. Then
  \begin{equation*}
    \sigma_c(A,B) \leq k^{3/2}\sqrt m.
  \end{equation*}
\end{lemma}
\begin{proof}
  For each colour $s\in S_c(A,B)$, choose an edge of colour $s$ and
  let $Y$ be the set of all chosen edges. Define $H$ to be the
  spanning subgraph of $K_n$ with edge set $Y$. Observe that $\size Y
  = \sigma_c(A,B)$. We claim that every two vertices $x,y \in
  B$ have fewer than $m$ common neighbors in the graph $H$. 

  For any two vertices $x,y\in B$, let $A_{xy}$ be the set of common
  neighbours of $x$ and $y$ in the graph $H$. Consider $z_1,z_2\in
  A_{xy}$. Since the induced subgraph of $K_n$ on $\Setx{x,y,z_1,z_2}$
  is not rainbow, we must have $c(xy) = c(z_1z_2)$. But then all the
  edges on $A_{xy}$ have colour $c(xy)$. Since $G$ contains no
  monochromatic complete subgraph of order $m$, we have $\size{A_{xy}}
  \leq m-1$ for every $x,y\in B$.

  Let $N$ be the number of all triples $xyz$ with $x,y\in B$ and $z\in
  A_{xy}$. By the above, 
  \begin{equation*}
    N \leq (m-1){\size B \choose 2}.
  \end{equation*}
  On the other hand, if we set $d_1,\dots,d_\ell$ to be the degrees of
  the vertices in $A$ in the graph $H$ ($\ell=\size A$), we find that
  $N$ equals the sum of $d_i \choose 2$ and therefore
  \begin{equation}\label{eq:upper-di}
    \sum_{i=1}^\ell {d_i \choose 2} \leq (m-1){k \choose 2}.
  \end{equation}

  Since the function $f(x) = x(x-1)/2$ is convex, we may use Jensen's
  inequality to derive
  \begin{equation*}
    \sum_{i=1}^\ell {d_i \choose 2} \geq \ell \cdot 
    \frac{(\sum_{i=1}^\ell d_i)/\ell \cdot ((\sum_{i=1}^\ell
        d_i)/\ell - 1)}2.
  \end{equation*}
  Observing that the sum of the $d_i$ is $\size Y$ and combining with
  (\ref{eq:upper-di}), we obtain
  \begin{equation*}
    \size Y(\size Y - \ell) \leq k(k-1)(m-1)\ell.
  \end{equation*}
  Furthermore, $\ell$ may be replaced with $k$ on both sides of the
  inequality as $\ell\leq k$. This leads to the following quadratic
  inequality in $\size Y$:
  \begin{equation}\label{eq:quad}
    {\size Y}^2 - k\size Y - k^2(k-1)(m-1) \leq 0.
  \end{equation}
  Solving~(\ref{eq:quad}), we find
  \begin{equation}\label{eq:size-y}
    \size Y \leq k \cdot \frac{1 + \sqrt{1+4(k-1)(m-1)}}2.
  \end{equation}

  The fraction in the right hand side of~(\ref{eq:size-y}) is easily
  seen to be at most $\sqrt{km}$ by a direct calculation, so $\size Y
  \leq k^{3/2}\sqrt m$ and the statement of the lemma is true.
\end{proof}

It is now easy to derive our upper bound on $\maxr n {K_m} {K_4}$:

\begin{proof}[Proof of Theorem~\ref{t:new-upper}]
  We proceed by induction on $n$. It is easy to check that for $n\leq
  21$, $n \choose 2$ is less than $n^{3/2}\sqrt{2m}$ for $m\geq 3$, so
  we may assume that $n \geq 22$. Set $\alpha=(1+1/22)^{3/2}$ and note
  that $(n+1)^{3/2} \leq \alpha n^{3/2}$.
  
  Let $c$ be an admissible colouring of $K_n$. For $X\subset V(K_n)$,
  define $\ell(X)$ as the number of colours used for edges on $X$. We
  need to prove that $\ell(V(K_n)) \leq n^{3/2}\sqrt{2m}$. To this
  end, partition $V(K_n)$ arbitrarily into sets $A$ and $B$ such that
  $\size A \leq n/2$ and $\size B \leq (n+1)/2$. By
  Lemma~\ref{l:across} and the induction, we then have
  \begin{align*}
    \ell(V(K_n)) &\leq \ell(A) + \ell(B) + \sigma_c(A,B)\\
    &\leq \Bigl(\frac n2\Bigr)^{3/2}\sqrt{2m} +
    \Bigl(\frac{n+1}2\Bigr)^{3/2}\sqrt{2m}
    + \Bigl(\frac{n+1}2\Bigr)^{3/2}\sqrt m\\
    &\leq n^{3/2}\sqrt m \cdot \frac{\alpha(\sqrt2+1)+\sqrt2}{2\sqrt2}\\
    &< n^{3/2}\sqrt{2m}.
  \end{align*}
\end{proof}


\section{Lower bounds}
\label{sec:lower}

Theorem~\ref{t:new-upper} improves the asymptotic upper bound for
$\maxr n {K_4} {K_4}$ to $O(n^{3/2})$, but this is still far from the
lower bound $n \log n$ of (\ref{eq:lower}). We now discuss a possible
way to improve the lower bound, which is based on incidence graphs of
finite projective planes. (See, e.g., \cite{bib:Cam-finite} for
background on finite geometries.)

Throughout this section, let $q$ be a prime power and
$n(q)=2(q^2+q+1)$.  Recall that there is a projective plane $PG(2,q)$
of order $q$. The incidence graph of $PG(2,q)$ is a $(q+1)$-regular
bipartite graph $L_q$ whose vertices are the points and the lines of
$PG(2,q)$, and whose edges join each point $p$ to the lines containing
$p$. Since $PG(2,q)$ has $q^2+q+1$ points and the same number of
lines, we can (and will) consider $L_q$ as a spanning subgraph of the
complete graph on $n(q)$ vertices.

One way to obtain an admissible colouring of $K_{n(q)}$ using
$\Omega(n^{3/2})$ colours is to first colour $L_q$, assigning each of
its edges a colour of its own (one that does not appear on any other
edge of $L_q$), and then try to extend this colouring to an admissible
colouring of $K_{n(q)}$. Since $L(q)$ has $\Omega(n(q)^{3/2})$ edges,
the number of colours is as requested.

Among the colourings obtained this way, we looked for ones satisfying
a mild additional restriction (which may make them somewhat easier to
find). Call a colouring $c$ of $K_{n(q)}$ \emph{special} if no edge of
$L_q \subset K_{n(q)}$ has a colour which is used on another edge of
$K_{n(q)}$. Note that to describe the colouring up to a permutation of
colours, it suffices to specify the colours of the edges not in $L_q$.

Figure~\ref{fig:fano} shows that special colourings do exist in the
case $q=2$, where we obtain the well-known Heawood graph on 14
vertices as the graph $L_2$. One method to find such colourings is as
follows. Regarding the vertices of $K_{14}$ as points and lines of
$PG(2,2)$, choose a 7-cycle $C \subset K_{14}$ on the points and a
7-cycle $C' \subset K_{14}$ on the lines such that every edge of $C$
and every edge of $C'$ are at distance $1$ in $L_2$. (It is not
difficult to show that such a choice is possible.) Assign colour $0$
to all edges of $K_{14}$ that are included in $C$ or $C'$, or join a
point of $PG(2,2)$ to a line. Colour the other edges of $K_{14}$ with
colour $1$. Easy case analysis confirms that the associated special
colouring is indeed admissible.

\begin{figure}
  \centering
  \begin{tikzpicture}[scale=3]
    \graph
    \tikzstyle{every node}+=[minimum size=5pt];
    \tikzstyle{dummy}+=[draw=none,fill=none];
    \newcommand{\PP}[1] {90 - 360/7 * {(#1)}:1}
    \newcommand{\LL}[1] {90 + 360/14 - 360/7 * {(#1)}:1}
    \foreach\i in {0,1,...,6}
    {
      \draw[ultra thick] (\PP\i) node (p) {}
        (\LL\i) node[dummy] {} edge (p)
        (\LL{\i+1}) node[dummy] {} edge (p)
        (\LL{\i+3}) node[dummy] {} edge (p);
      \draw[black!30] 
        (\PP{\i+2}) node[dummy] {} edge (p)
        (\PP{\i+3}) node[dummy] {} edge (p)
        (\PP{\i+4}) node[dummy] {} edge (p)
        (\PP{\i+5}) node[dummy] {} edge (p);
      \draw[black!30] (\LL\i) node[black] (l) {}
        (\LL{\i+1}) node[dummy] {} edge (l)
        (\LL{\i+2}) node[dummy] {} edge (l)
        (\LL{\i+5}) node[dummy] {} edge (l)
        (\LL{\i+6}) node[dummy] {} edge (l);
    }

  \end{tikzpicture}
  \caption{A special admissible colouring of $K_{14}$ by $23$
    colours. Thick edges are those of $L_2$ (hence each of them has a
    colour of its own, say from $\Setx{2,\dots,22}$), missing edges represent
    colour $0$, grey edges represent colour $1$.}
  \label{fig:fano}
\end{figure}
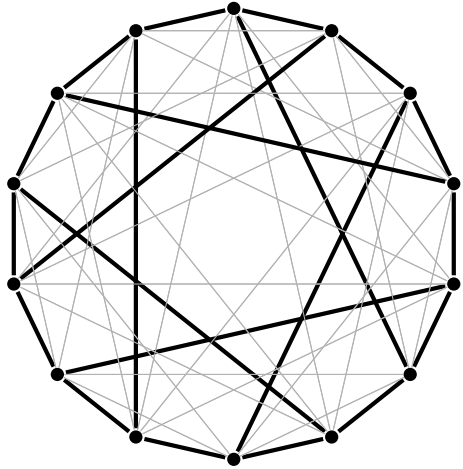

In general, a rotational symmetry such as that of
Figure~\ref{fig:fano} may be useful when looking for special
colourings. The first question to be addressed, however, is whether
the graphs $L_q$ ($q \geq 3$) themselves admit a rotationally
symmetric drawing. More precisely, let us call a Hamilton cycle
$v_0v_1\dots v_{n(q)-1}$ in $L_q$ \emph{rotational} if for each
$i,j\in \Setx{0,\dots,n(q)-1}$, $v_iv_j \in E(L_q)$ if and only if
$v_{i+2}v_{j+2} \in E(L_q)$ (indices taken modulo $n(q)$). Somewhat
surprisingly, it turned out that all the graphs $L_q$, where $q$ is a
prime power, have rotational Hamilton cycles. We suppose that this is
a known result, but since our search in the literature did not reveal
anything, we briefly sketch the proof. The only result in this
direction we are aware of is the result of
Brown~\cite{bib:Bro-hamiltonian} that the graphs $L_q$, for prime $q$,
are Hamiltonian. (We are indebted to Geoff Exoo for this information.)

\begin{proposition}\label{p:rotational}
  For any prime power $q$, the graph $L_q$ admits a rotational
  Hamilton cycle.
\end{proposition}
\begin{proof}
  The proof is inspired by a proof of Erd\H{o}s~\cite{bib:Erd-problem}
  concerning so-called Sidon sets (see also~\cite{bib:Cho-solution}),
  which in turn builds on a proof of Singer~\cite{bib:Sin-theorem} of
  the existence of perfect difference sets. We take a primitive
  element $\alpha$ in the field $GF(q^3)$ and view the field as a
  vector space $V$ of dimension 3 over $F = GF(q)$. Recall that points
  and lines of $PG(2,q)$ correspond one-to-one to subspaces of $V$ of
  dimension 1 and 2, respectively. In particular, for
  $i=0,\dots,q^2+q$, let $p_i$ be the point of $PG(2,q)$ corresponding
  to the line in $V$ through 0 and $\alpha^i$. All of these points are
  distinct, since $\alpha^j$ is a scalar multiple of $\alpha^i$ if and
  only if $j-i$ is a multiple of $q^2+q+1$.

  Let $\ell_i$ ($i=0,\dots,q^2+q-1$) be the unique line of $PG(2,q)$
  through $\alpha^i$ and $\alpha^{i+1}$, and similarly let
  $\ell_{q^2+q}$ be the line through $\alpha^{q^2+q}$ and 1.

  To verify that $H = (p_0,\ell_0,p_1,\dots,p_{q^2+q},\ell_{q^2+q})$
  is a Hamilton cycle in $L_q$, all we need to show is that all the
  lines $\ell_i$ are distinct. Suppose not. Then for some $0\leq i < j
  \leq q^2+q$, the set $\Setx{0, \alpha^i, \alpha^{i+1}, \alpha^j,
    \alpha^{j+1}}$ is contained in a plane $P$ of $V$ (note that this
  holds even in the boundary case $j=q^2+q$).

  Without loss of generality, we may assume that $i=0$ (multiplying
  the equation of $P$ by $\alpha^{-i}$ if necessary), which implies
  that the set $\Setx{0, 1,\alpha,\alpha^j,\alpha^{j+1}}$ is contained
  in $P$. Since 1 and $\alpha$ span $P$, we may write $\alpha^j =
  c\alpha + d$, where $c,d\in F$. Hence $\alpha^{j+1} = c\alpha^2 +
  d\alpha$, and since $\alpha^{j+1}$ is in $P$, so is
  $\alpha^2$. However, a similar argument then shows that $\alpha^3$
  and all the successive powers of $\alpha$ are also in $P$, a
  contradiction with the fact that $V$ is 3-dimensional.

  What remains to be shown is that the Hamilton cycle $H$ is
  rotational, i.e. that if a point $p_i$ lies on a line $\ell_j$, then
  $p_{i+1}$ lies on $\ell_{j+1}$. A key observation is that $p_i$ lies
  on $\ell_j$ if and only if $\alpha^i$, $\alpha^j$ and $\alpha^{j+1}$
  are linearly dependent in $V$. Assuming this condition holds, it is
  clear that $\alpha^{i+1}$, $\alpha^{j+1}$ and $\alpha^{j+2}$ are
  also linearly dependent, i.e. $p_{i+1}$ lies on $\ell_{j+1}$ as
  claimed.
\end{proof}

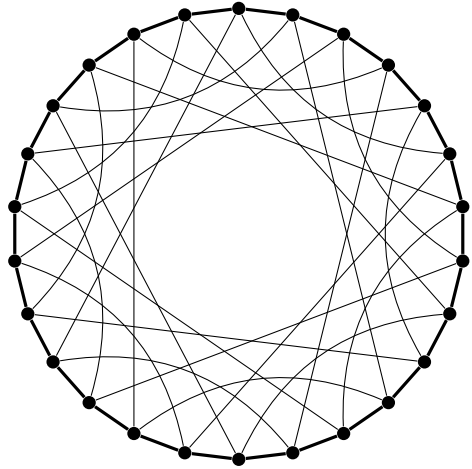
\begin{figure}
  \centering
  \begin{tikzpicture}[scale=3]
    \graph
    \tikzstyle{every node}+=[minimum size=5pt];
    \tikzstyle{dummy}+=[draw=none,fill=none];
    \newcommand{\VV}[1] {90 - 360/26 * {(#1)}:1}
    \foreach\i in {0,2,...,24}
    {
      \draw (\VV\i) node (v) {}
        (\VV{\i-1 + 2*0}) node {} edge[very thick] (v)
        (\VV{\i-1 + 2*1}) node[dummy] {} edge[very thick] (v)
        (\VV{\i-1 + 2*3}) node[dummy] {} edge[bend left] (v)
        (\VV{\i-1 + 2*9}) node[dummy] {} edge (v);
    }
  \end{tikzpicture}
  \caption{A rotational Hamilton cycle in $L_3$ (bold).}
  \label{fig:rot}
\end{figure}

An example of a rotational Hamilton cycle constructed using
Proposition~\ref{p:rotational} is shown in Figure~\ref{fig:rot}. 

Let $v_0v_1\dots v_{n(q)-1}$ be a rotational Hamilton cycle in
$L_q$ and let $c$ be an edge-colouring of $K_{n(q)}$ with colours in a
set $Y$. For $i\neq j$ in $\Setx{0,\dots,n(q)-1}$, we define a symbol
$\bar c_{i,j} \in Y\cup\Setx{\text{$\ast$}}$ by
\begin{equation*}
  \bar c_{i,j} =
  \begin{cases}
    * & \text{if $v_iv_j$ is an edge of $L_q$,}\\
    c(v_iv_j) & \text{otherwise.}
  \end{cases}
\end{equation*}
For $i=0,\dots,n(q)-1$, we define the words
\begin{equation*}
  c(v_i) = (\bar c_{i,i+1} \bar c_{i,i+2} \dots 
  \bar c_{i,i-1}),
\end{equation*}
where the indices are taken modulo $n(q)$. We extend the above
terminology and call the colouring $c$ \emph{rotational} if
$c(v_i)=c(v_0)$ for all even $i$, and $c(v_i) = c(v_1)$ for all odd
$i$.

This is the case in Figure~\ref{fig:fano}, where we have
\begin{align*}
  c(v_0) &= (\text{$\ast$001$\ast$1010100$\ast$}),\\
  c(v_1) &= (\text{$\ast$1010000$\ast$101$\ast$}).
\end{align*}

For $q = 3$, we found a number of rotational colourings by a computer
search. One of these, for instance, is determined by the words
\begin{align*}
  c(v_0) &= (\text{$\ast$00001$\ast$001$\ast$1110100110010$\ast$}),\\
  c(v_1) &= (\text{$\ast$0100110010111$\ast$100$\ast$10000$\ast$}).
\end{align*}
Note that the words in the latter case have the additional curious
property that $c(v_1)$ is the reverse of $c(v_0)$.

In general, we had to leave the following problem open:
\begin{problem}\label{prob:lower}
  For $q\geq 3$, are there any admissible rotational colourings of
  $K_{n(q)}$? Are there any admissible special colourings?
\end{problem}

We think that even a negative answer to Problem~\ref{prob:lower} may
shed some light on the question whether the upper bound given in
Theorem~\ref{t:new-upper} is asymptotically tight.


\section*{Acknowledgment}

We are indebted to Martin Klazar for pointing us to the results on
Sidon sets and perfect difference sets which led us to the proof of
Proposition~\ref{p:rotational}. Geoff Exoo kindly informed us about
the paper of Brown~\cite{bib:Bro-hamiltonian}.


\end{document}